\documentclass[12pt,reqno]{amsart}
\usepackage{amscd,amsmath,amsthm,amssymb}
\usepackage{color}
\usepackage{pstricks}
\usepackage{stmaryrd}
\usepackage{tikz}
\usepackage{url}

\usepackage{latexsym}
\usepackage{amsfonts,amsmath,mathtools}
\usepackage{graphics}
\usepackage{float}
\usepackage{enumitem}

\usepackage{booktabs} 
\usepackage{colortbl}
\usepackage{lipsum}

\newpsstyle{fatline}{linewidth=1.5pt}
\newpsstyle{fyp}{fillstyle=solid,fillcolor=verylight}
\definecolor{verylight}{gray}{0.97}
\definecolor{light}{gray}{0.9}
\definecolor{medium}{gray}{0.85}
\definecolor{dark}{gray}{0.6}

\def\NZQ{\mathbb}               

\def\QQ{{\NZQ Q}}
\def\ZZ{{\NZQ Z}}

\def\frk{\mathfrak}               

\def\mm{{\frk m}}

\def\G{{\mathcal G}}

\def\pd{\textup{proj}\phantom{.}\!\textup{dim}}

\def\opn#1#2{\def#1{\operatorname{#2}}} 
\opn\chara{char} \opn\length{\ell} \opn\pd{pd} \opn\rk{rk}
\opn\projdim{proj\,dim} \opn\injdim{inj\,dim} \opn\rank{rank}
\opn\depth{depth} \opn\grade{grade} \opn\height{height}
\opn\embdim{emb\,dim} \opn\codim{codim}

\opn\Tr{Tr} \opn\bigrank{big\,rank}
\opn\superheight{superheight}\opn\lcm{lcm}
\opn\trdeg{tr\,deg}
	\opn\reg{reg} \opn\lreg{lreg} \opn\ini{in} \opn\lpd{lpd}
	\opn\size{size} \opn\sdepth{sdepth}
	\opn\link{link}\opn\fdepth{fdepth}\opn\lex{lex}
	\opn\tr{tr}
	\opn\type{type}
	\opn\gap{gap}
	\opn\diam{diam}
	\opn\Mod{Mod}
	\opn\div{div} \opn\Div{Div} \opn\cl{cl} \opn\Cl{Cl}
	\opn\Spec{Spec} \opn\Supp{Supp} \opn\supp{supp} \opn\Sing{Sing}
	\opn\Ass{Ass} \opn\Min{Min}\opn\Mon{Mon}
	\opn\Ann{Ann} \opn\Rad{Rad} \opn\Soc{Soc}
	\opn\Im{Im} \opn\Ker{Ker} \opn\Coker{Coker} \opn\Am{Am}
	\opn\Hom{Hom} \opn\Tor{Tor} \opn\Ext{Ext} \opn\End{End}
	\opn\Aut{Aut} \opn\id{id}
	
	\opn\nat{nat}
	\opn\pff{pf}
	\opn\Pf{Pf} \opn\GL{GL} \opn\SL{SL} \opn\mod{mod} \opn\ord{ord}
	\opn\Gin{Gin} \opn\Hilb{Hilb}\opn\sort{sort}
	\opn\PF{PF}\opn\Ap{Ap}
	\opn\dist{dist}
	\opn\aff{aff}
	\opn\relint{relint} \opn\st{st}
	\opn\lk{lk} \opn\cn{cn} \opn\core{core} \opn\vol{vol}  \opn\inp{inp} \opn\nilpot{nilpot}
	\opn\link{link} \opn\star{star}\opn\lex{lex}\opn\set{set}
	\opn\width{wd}
	\opn\Fr{F}
	\opn\QF{QF}
	\opn\G{G}
	\opn\type{type}\opn\res{res}
	\opn\conv{conv}
	\opn\sr{sr}
	\opn\gr{gr}
	
	\def\pot#1#2{#1[\kern-0.28ex[#2]\kern-0.28ex]}

	\opn\dirlim{\underrightarrow{\lim}}
	\opn\inivlim{\underleftarrow{\lim}}
	\let\tensor=\otimes

	\def\Implies{\ifmmode\Longrightarrow \else
		\unskip${}\Longrightarrow{}$\ignorespaces\fi}
	\def\implies{\ifmmode\Rightarrow \else
		\unskip${}\Rightarrow{}$\ignorespaces\fi}
	\def\iff{\ifmmode\Longleftrightarrow \else
		\unskip${}\Longleftrightarrow{}$\ignorespaces\fi}

	\let\:=\colon
	\newtheorem{Theorem}{Theorem}[section]
	\newtheorem{Lemma}[Theorem]{Lemma}
	\newtheorem{Corollary}[Theorem]{Corollary}
	\newtheorem{Proposition}[Theorem]{Proposition}
	\theoremstyle{definition}
	\newtheorem{Remark}[Theorem]{Remark}
	
	\newtheorem{Example}[Theorem]{Example}
	\newtheorem{Examples}[Theorem]{Examples}

	\newtheorem{Conjecture}[Theorem]{Conjecture}
	\newtheorem{Question}[Theorem]{Question}
	
	\let\epsilon\varepsilon
	\let\kappa=\varkappa
	\def\qed{\ifhmode\textqed\fi
		\ifmmode\ifinner\hfill\quad\qedsymbol\else\dispqed\fi\fi}
	\def\textqed{\unskip\nobreak\penalty50
		\hskip2em\hbox{}\nobreak\hfill\qedsymbol
		\parfillskip=0pt \finalhyphendemerits=0}
	\def\dispqed{\rlap{\qquad\qedsymbol}}
	
	\opn\dis{dis}
	\def\pnt{{\raise0.5mm\hbox{\large\bf.}}}
	
	\opn\Lex{Lex}
	\opn\Max{Max}
	\opn\Shad{Shad}
	\opn\astab{astab}
	\def\p{\mathfrak{p}}
	\def\q{\mathfrak{q}}
	\def\m{\mathfrak{m}}
	
	\opn\v{v}

	\def\vstab{\textup{v-stab}}

\begin{document}

\title{The $\textup{v}$-function of powers of sums of ideals}	
\author{Antonino Ficarra, Pedro Macias Marques}

\address{Antonino Ficarra, Departamento de Matem\'{a}tica, Escola de Ci\^{e}ncias e Tecnologia, Centro de Investiga\c{c}\~{a}o, Matem\'{a}tica e Aplica\c{c}\~{o}es, Instituto de Investiga\c{c}\~{a}o e Forma\c{c}\~{a}o Avan\c{c}ada, Universidade de \'{E}vora, Rua Rom\~{a}o Ramalho, 59, P--7000--671 \'{E}vora, Portugal}
\email{antonino.ficarra@uevora.pt\,\,\,\,\,\,\,\,antficarra@unime.it}

\address{Pedro Macias Marques, Departamento de Matem\'{a}tica, Escola de Ci\^{e}ncias e Tecnologia, Centro de Investiga\c{c}\~{a}o, Matem\'{a}tica e Aplica\c{c}\~{o}es, Instituto de Investiga\c{c}\~{a}o e Forma\c{c}\~{a}o Avan\c{c}ada, Universidade de \'{E}vora, Rua Rom\~{a}o Ramalho, 59, P--7000--671 \'{E}vora, Portugal}
\email{pmm@uevora.pt}

\thanks{
}



\maketitle
\vspace*{-0.4cm}
\begin{abstract}
	Let $K$ be a field, $I\subset R=K[x_1,\dots,x_n]$ and $J\subset T=K[y_1,\dots,y_m]$ be graded ideals. Set $S=R\otimes_KT$ and let $L=IS+JS$. The behaviour of the $\textup{v}$-function $\textup{v}(L^k)$ in terms of the $\textup{v}$-functions $\textup{v}(I^k)$ and $\textup{v}(J^k)$ is investigated. When $I$ and $J$ are monomial ideals, we describe $\textup{v}(L^k)$, giving an explicit formula involving the local $\textup{v}$-numbers $\textup{v}_{\mathfrak{p}}(I^k)$ and $\textup{v}_{\mathfrak{q}}(J^k)$.
\end{abstract}

\section{Introduction}

Let $R=K[x_1,\dots,x_n]$ be the standard graded polynomial ring over a field $K$ with maximal ideal $\m=(x_1,\dots,x_n)$, and let $I\subset R$ be a graded ideal. In \cite{CSTVV20} Cooper \emph{et al.\ } introduced a new invariant associated to $I$ which they called the $\v$-number of $I$, in honor of the mathematician Wolmer Vasconcelos. Let $\p\in\Ass(I)$ be an associated prime. Since $I$ is graded, $\p$ is a graded prime ideal and there exists an homogeneous polynomial $f\in R$ such that $(I:f)=\p$. The \textit{$\v_\p$-number} of $I$ is the least degree of such an element $f$. The \textit{$\v$-number} of $I$ is defined as $\v(I)=\min_{\p\in\Ass(I)}\v_\p(I)$.

By Brodmann \cite{B79}, see also \cite{DMNB23}, the set $\Ass(I^k)$ stabilizes: $\Ass(I^{k+1})=\Ass(I^k)$ for all $k\gg0$. We denote the common sets $\Ass(I^k)$ for $k\gg0$ by $\Ass^\infty(I)$.

In \cite[Theorem 3.1]{FS2}, and independently by Conca in \cite[Theorem 1.1]{Conca23}, it was proved that the functions $k\mapsto\v_\p(I^k)$, for $\p\in\Ass^\infty(I)$, and $k\mapsto\v(I^k)$ are linear for $k\gg0$. By \cite[Theorem 4.1]{FS2}, we have $\v(I^k)=\alpha(I)k+b$ for $k\gg0$, where $b\in\ZZ$ and $\alpha(I)=\min\{d:(I/\m I)_d\ne0\}$ is the \textit{initial degree} of $I$. These results inspired many papers \cite{BM23,BMS24,Conca23,DJS24,F2023,FS2,FSPackA,Fior24,Ghosh24,KS23,KNS24,S2023,VS24}, and a \textit{Macaulay2} package \cite{FSPack}.

Let $G$ be a finite simple graph on the vertex set $V(G)=\{x_1,\dots,x_n\}$ and edge set $E(G)$. The ideal $I(G)=(x_ix_j:\{x_i,x_j\}\in E(G))$ is called the \textit{edge ideal} of $G$. A beautiful result of Biswas, Mandal, and Saha \cite[Theorem 5.1]{BMS24} guarantees that $\v(I(G)^k)=2k-1$ for all $k\ge|E(G)|+1$ if $G$ is a connected graph. If $G$ is an arbitrary graph with $c$ connected components $G_1,\dots,G_c$, then one could expect that $\v(I(G)^k)=2k+(c-2)$ for all $k\gg0$. We show in Corollary \ref{Cor:v(I(G)^k)} that this is indeed the case. Notice that $I(G)=I(G_1)+\dots+I(G_c)$ and the generators of the ideals $I(G_i)$, $i=1,\dots,c$, are monomials in pairwise disjoint sets of variables.

This fact suggest to study the following more general problem. Let $K$ be a field, and let $I\subset R=K[x_1,\dots,x_n]$ and $J\subset T=K[y_1,\dots,y_m]$ be proper homogeneous ideals. Set $S=R\otimes_K T=K[x_1,\dots,x_n,y_1,\dots,y_m]$. By abuse of notation we denote the extended ideals $IS$ ans $JS$ again by $I$ and $J$.
\begin{Question}\label{Ques:v-I-J}
	What can we say about the $\v$-functions $\v((IJ)^k)$ and $\v((I+J)^k)$ in terms of the $\v$-functions $\v(I^k)$ and $\v(J^k)$?
\end{Question}

To address Question \ref{Ques:v-I-J} will be our main concern. The result in Corollary \ref{Cor:v(I(G)^k)}, which was our original motivation for writing this paper, will be a consequence of the more general theory we are going to develop.

Now, we outline the main contents and findings of this paper. In Section \ref{Sec:2-A-P} we address and completely solve Question \ref{Ques:v-I-J} in the case of the product $(IJ)^k$. It is noted in Lemma \ref{Lem:simpleAss} that $\Ass((IJ)^k)=\Ass(I^k)\cup\Ass(J^k)$ for all $k\ge1$. Hence, in Theorem \ref{Thm:Prod} the local $\v$-functions of $(IJ)^k$ are calculated and in Corollary \ref{Cor:Prod} we prove that $\v((IJ)^k)= \min\{\v(I^k)+\alpha(J)k,\,\v(J^k)+\alpha(I)k\}$ for $k\gg0$. These results and the general computation of the $\v$-number rest on the following result due to Conca \cite{Conca23}. Denote by $\Max(I)$ the set of prime ideals of $\Ass(I)$ maximal with respect to inclusion. For an homogeneous ideal $I\subset S$ and $\p\in\Ass(I)$, set
\begin{equation}\label{eq:X_p}
	X_\p\ =\ \begin{cases}
	\displaystyle(\prod_{\p_1\in\Ass(I)\ :\ \p\subsetneq\p_1}\p_1)&\textit{if}\ \ \p\in\Ass(I)\setminus\Max(I),\\
	\hfill S&\hfill\textit{otherwise}.
\end{cases}
\end{equation}

\begin{Lemma}\label{Lem:Conca}
	\textup{\cite[Lemma 1.2]{Conca23}} Let $I\subset S$ be a graded ideal and $\p\in\Ass(I)$. Then,
	$$
	\v_\p(I)\ =\ \alpha\biggl(\frac{(I:\p)}{I:(\p+X_\p^\infty)}\biggr).
	$$
	Here $I:(\p+X_\p^\infty)=\bigcup_{k\ge0}(I:(\p+X_\p^k))$.
\end{Lemma}

Regarding the $\v$-function $\v((I+J)^k)$, we do not have yet a complete solution to Question \ref{Ques:v-I-J}. But we can provide a satisfactory answer when $I$ and $J$ are monomial ideals, as we describe next.

We take a moment to justify why we confine ourselves only to monomial ideals. From the computational point of view, it is reasonable to restrict our attention only to monomial ideals. Indeed, the associated primes of a monomial ideal are monomial prime ideals \cite[Corollary 1.3.9]{HH2011}, that is, ideals generated by a subset of the variables. Thus, if $L\subset S$ is a monomial ideal, there are only finitely many candidates that can be prime ideals $\p\in\Ass^\infty(L)$, precisely the $2^{\dim S}-1=2^{n+m}-1$ prime ideals $\p_A=(z:z\in A)$ where $A\subseteq\{x_1,\dots,x_n,y_1,\dots,y_m\}$ is non-empty. For each $\p_A$, Theorem \ref{Thm:pinAss8(I)} is an efficient computational criterion, which was implemented in the \textit{Macaulay2} \cite{GDS} package \texttt{VNumber} \cite{FSPack}, to check if $\p_A\in\Ass^\infty(L)$. This result was firstly noted by Bayati, Herzog, and Rinaldo in \cite[Section 2]{BHR12} and implemented in \cite{BHR11}, and takes advantage of the so-called \textit{monomial localization}. Ordinary localization would destroy the multigraded structure of $L$. On the other hand, when $L$ and $\p$ are arbitrary homogeneous ideals of $S$, one can not use monomial localization and Theorem \ref{Thm:pinAss8(I)} becomes inefficient from the computational point of view as pointed out in \cite[Section 4]{FSPackA}. Moreover, for a general homogeneous ideal $L$, a priori there is no obvious finite set of homogeneous prime ideals which contains $\Ass^\infty(L)$. In \cite{KS19} Kim and Swanson even provided examples of prime ideals of $S$ whose number of associated primes of powers are exponential in the number of variables of $S$.

Coming back to the computation of $\v((I+J)^k)$, notice that for $k\gg0$
$$
\v((I+J)^k)\ =\ \min_{\p\in\Ass^\infty(I+J)}\v_\p((I+J)^k).
$$
Thus, to study this function one has to know the set $\Ass^\infty(I+J)$ in terms of the sets $\Ass^\infty(I)$ and $\Ass^\infty(J)$. This problem was already solved by Nguyen and Tran \cite[Theorem 4.5]{NguyenT}. For an ideal $L\subset S$, we set $\Ass^*(L)=\bigcup_{k\ge1}\Ass(L^k)$. By Brodmann $\Ass^*(L)$ is a finite set \cite{B79}. Nguyen and Tran's Theorem (here Theorem \ref{Thm:AssInfty}) says that
$$
\Ass^\infty(I+J)\ =\ \bigcup_{\substack{\p\in\Ass^*(I)\\ \q\in\Ass^\infty(J)}}\{\p+\q\}\cup\bigcup_{\substack{\p\in\Ass^\infty(I)\\ \q\in\Ass^*(J)}}\{\p+\q\}.
$$
The proof of Nguyen and Tran is based on the computation by H\`{a}, Trung, and Trung of some filtrations of $S/(I+J)^k$ and a well-known formula for $\Ass_S(M\otimes_K N)$ where $M$ is a $R$-module and $N$ is a $T$-module \cite{HTT16}. In Section \ref{Sec:3-A-P} we offer a new and relatively easy proof of this result when $I$ and $J$ are monomial ideals. The key ingredients of this new proof are Theorem \ref{Thm:pinAss8(I)} which was showed in \cite[Theorem 2.1]{FSPackA} and some easy basic facts from Hilbert function and Krull dimension theories. These facts are recalled in Remarks \ref{Rem:elementary1}, \ref{Rem:elementary2} and \ref{Rem:elementary3}. At the moment we do not know if our argument can be extended to all homogeneous ideals $I\subset R$ and $J\subset T$.

In Section \ref{Sec:4-A-P} we prove a formula for the local $\v$-numbers of $(I+J)^k$ for all $k\ge1$, when $I$ and $J$ are monomial ideals (Theorem \ref{Thm:FormulaV-Num-MonId}). For $k=1$ we have $\v(I+J)=\v(I)+\v(J)$, see \cite[Proposition 3.9]{SS20} due to Saha and Sengupta. Due to some experimental evidence we could expect that the formula given in Theorem \ref{Thm:FormulaV-Num-MonId} is valid for any graded ideals $I\subset R$ and $J\subset T$. Some results by Ambhore, Saha, and Sengupta \cite{ASS23} support this expectation.

In \cite[Question 5.1]{FS2} it is asked whether the inequality $\v(L^k)<\reg(L^k)$ holds for all $k\gg0$ and all graded ideals $L\subset S$. In Corollary \ref{Cor:v-reg-prod} we show that to solve this question we may assume that $L$ can not be written as a product $L=IJ$, where $I$ and $J$ are homogeneous ideals whose generators are polynomials in pairwise disjoint sets of variables. In Corollary \ref{Cor:regIneq} we show that in the monomial case, we can also assume that $L$ can not be written as the sum $I+J$ of two proper monomial ideals $I$ and $J$ whose generators are monomials in pairwise disjoint sets of variables.

Finally, Section \ref{Sec:5-A-P} contains several applications of Theorem \ref{Thm:FormulaV-Num-MonId}. It was proved in \cite[Proposition 2.2]{F2023} that for any monomial ideal $L\subset S$ we have the natural lower bound $\v(L^k)\ge\alpha(L)k-1$ for all $k\ge1$. Let $I_1,\dots,I_t\subset S$ be monomial ideals whose generators are monomials in pairwise disjoint sets of variables. Assuming that $\v(I_j^k)=\alpha(I_j)k-1$ for all $j=1,\dots,t$ and all $k\gg0$, in Theorem \ref{Thm:DisjointSupports} we provide for all $k\ge1$ a natural lower bound of $\v((I_1+\dots+I_t)^k)$. Under some additional assumptions this lower bound is achieved. This result is applied to edge ideals, monomial complete intersections and vertex splittable monomial ideals.

\section{The $\v$-function of products}\label{Sec:2-A-P}

Here and in the next section, we maintain the following notation.

Let $K$ be a field, and let $I\subset R=K[x_1,\dots,x_n]$ and $J\subset T=K[y_1,\dots,y_m]$ be proper homogeneous ideals. Let $S=R\otimes_K T=K[x_1,\dots,x_n,y_1,\dots,y_m]$. Notice that for any ideal $L\subset R$ we have $\Ass(L)=\Ass(LS)$ and for any ideal $H\subset T$ we have $\Ass(H)=\Ass(HS)$. This fact justifies our next notational choice. With abuse of notation we denote the extended ideals $IS$ and $JS$ again by $I$ and $J$.\smallskip

The purpose of this section is to compute the (local) $\v$-function of $IJ$ in terms of the (local) $\v$-functions of $I$ and $J$. We begin with the next elementary observation.
\begin{Lemma}\label{Lem:simpleAss}
	With the notation introduced, for all $k\ge1$ we have
	$$
	\Ass((IJ)^k)\ =\ \Ass(I^k)\cup\Ass(J^k).
	$$
	In particular, $\Ass^\infty(IJ)\ =\ \Ass^\infty(I)\cup\Ass^\infty(J)$.
\end{Lemma}
\begin{proof}
	Note that $(IJ)^k = I^kJ^k = I^k\cap J^k$. The short exact sequence
	$$
	0\rightarrow S/(I^kJ^k)\rightarrow S/I^k\oplus S/J^k\rightarrow S/(I^k+J^k)\rightarrow 0
	$$
	implies that $\Ass((IJ)^k)\ \subseteq \ \Ass(I^k)\cup\Ass(J^k)$ (see \cite[Lemma~3.6]{Eis94}).
	
	For the opposite inclusion, let $\p\in\Ass(I^k)$. Then, there exists $f\in R$ such that $(I^k:f)=\p$. Let $g\in J^k\cap T$. Since $(IJ)^k=I^kJ^k=I^k\cap J^k$, then
	\begin{equation}\label{eq:colon1}
		((IJ)^k:fg)\ =\ ((I^k\cap J^k):fg)\ =\ (I^k:fg)\cap(J^k:fg).
	\end{equation}
	Notice that
	\begin{equation}\label{eq:colon2}
		(I^k:fg)\ =\ ((I^k:f):g)\ =\ (\p:g)\ =\ \p,
	\end{equation}
	because $g\in T$ while $\p$ is an ideal whose generators are polynomials of $R$. Since $fg\in J^k$ because $g\in J^k$, then $(J^k:fg)=S$. This fact and equations (\ref{eq:colon1}) and (\ref{eq:colon2}) imply that $((IJ)^k:fg)=\p$. Hence $\p\in\Ass((IJ)^k)$. The case $\q\in\Ass(J^k)$ is analogous.
\end{proof}

\begin{Theorem}\label{Thm:Prod}
	With the notation introduced, we have
	\begin{enumerate}
		\item[\textup{(a)}] For all $k\ge1$ and all $\p\in\Ass(I^k)$, we have $\v_\p((IJ)^k)=\v_\p(I^k)+\alpha(J)k$.
		
		\smallskip
		\item[\textup{(b)}] For all $k\ge1$ and all $\q\in\Ass(J^k)$, we have $\v_\q((IJ)^k)=\v_\q(J^k)+\alpha(I)k$.
	\end{enumerate}
\end{Theorem}
\begin{proof}
	Let $\p\in\Ass(I^k)$. Since $I$ and $J$ are ideals whose generators are in pairwise disjoint sets of variables, $\p\subset R$ and $\Ass((IJ)^k)=\Ass(I^k)\cup\Ass(J^k)$, we have that $X_\p\subset R$, where $X_\p$ is defined in (\ref{eq:X_p}). By Lemma \ref{Lem:Conca} and what we observed,
	\begin{align*}
		\v_\p((IJ)^k)\ &=\ \alpha\big(\frac{((IJ)^k:\p)}{(IJ)^k:(\p+X_\p^\infty)}\big)\ =\ \alpha\big(\frac{J^k(I^k:\p)}{J^k(I^k:(\p+X_\p^\infty))}\big)\\
		&=\ \alpha\big(\frac{(I^{k}:\p)}{(I^{k}:(\p+X_\p^\infty))}\big)+\alpha(J^{k})\\
		&=\ \v_\p(I^k)+\alpha(J)k.
	\end{align*}
	Statement (a) follows. The proof of statement (b) is analogous.
\end{proof}

\begin{Corollary}\label{Cor:Prod}
	With the notation introduced, for all $k\gg0$ we have
	$$
	\v((IJ)^k)\ =\ \min\{\v(I^k)+\alpha(J)k,\,\v(J^k)+\alpha(I)k\}.
	$$
\end{Corollary}

\begin{Corollary}\label{Cor:v-reg-prod}
	With the notation introduced, suppose that $\v(I^k)<\reg(I^k)$ and $\v(J^k)<\reg(J^k)$ for all $k\gg0$. Then $\v((IJ)^k)<\reg((IJ)^k)$ for all $k\gg0$.
\end{Corollary}
\begin{proof}
	Notice that $\reg(I^k)\ge\alpha(I)k$ and $\reg(J^k)\ge\alpha(J)k$ for all $k\ge1$. Moreover, by \cite[Corollary 3.2]{HRR} we have $\reg((IJ)^k)=\reg(I^k)+\reg(J^k)$. Using the assumption, we obtain that $\reg((IJ)^k)>\v(I^k)+\alpha(J)k$ and also $\reg((IJ)^k)>\alpha(I)k+\v(J^k)$ for all $k\gg0$. Applying Corollary \ref{Cor:Prod} the assertion follows.
\end{proof}

\section{The set of stable primes of sums of monomial ideals}\label{Sec:3-A-P}

Let $K$ be a field, and let $I\subset R=K[x_1,\dots,x_n]$ and $J\subset T=K[y_1,\dots,y_m]$ be proper homogeneous ideals. Set $S=R\otimes_K T=K[x_1,\dots,x_n,y_1,\dots,y_m]$.

We set $\Ass^*(I)=\bigcup_{k\ge1}\Ass(I^k)$. By Brodmann, $\Ass^*(I)$ is a finite set \cite{B79}. The following result was shown by Nguyen and Tran \cite[Theorem 4.5]{NguyenT}.

\begin{Theorem}\label{Thm:AssInfty}
	With the notation introduced, we have
	$$
	\Ass^\infty(I+J)\ =\ \bigcup_{\substack{\p\in\Ass^*(I)\\ \q\in\Ass^\infty(J)}}\{\p+\q\}\cup\bigcup_{\substack{\p\in\Ass^\infty(I)\\ \q\in\Ass^*(J)}}\{\p+\q\}.
	$$
\end{Theorem}

The proof of \cite[Theorem 4.5]{NguyenT} follows by lengthy computations involving some filtrations of $S/(I+J)^k$ and a well-known formula for $\Ass_S(M\otimes_K N)$ where $M$ is a $R$-module and $N$ is a $T$-module \cite[Theorem 2.5]{HNTT20}.

In this section, we provide a more elementary proof when $I$ and $J$ are monomial ideals. The new argument we use is based on the following theorem.
\begin{Theorem}\label{Thm:pinAss8(I)}
	\textup{\cite[Theorem 2.1]{FSPackA}} Let $R=K[x_1,\dots,x_n]$ be a standard graded polynomial ring over a field $K$. Let $I\subset R$ be a homogeneous ideal and $\p\in\Spec(R)$. The following conditions are equivalent:
	\begin{enumerate}
		\item[\textup{(a)}] $\p\in\Ass^\infty(I)$.
		\item[\textup{(b)}] The Krull dimension $\dim H_{\mu(\p R_\p)-1}(\p R_\p\,;\,\mathcal{R}(I)_\p)$ is positive.
	\end{enumerate}
\end{Theorem}

Here:
\begin{enumerate}
	\item[-] $\mathcal{R}(I)=\bigoplus_{k\ge0}I^k$ is the Rees algebra of $I$,\medskip
	\item[-] $\mu(\p R_\p)=\dim_{R_\p/\p R_\p}(\p R_\p/\p^2 R_\p)$ is the minimal number of generators of $\p R_\p$,\medskip
	\item[-] and $H_{\mu(\p R_\p)-1}(\p R_\p\,;\,\mathcal{R}(I)_\p)$ is the $(\mu(\p R_\p)-1)$th Koszul homology module of $\p R_\p$ with respect to $\mathcal{R}(I)_\p$.
\end{enumerate}

It is noted in the proof of \cite[Theorem 2.1]{FSPackA} that
\begin{equation}\label{eq:KoszHom}
	H_{\mu(\p R_\p)-1}(\p R_\p\,;\,\mathcal{R}(I)_\p)\ =\ \bigoplus_{k\ge0}\frac{(I_\p^{k}:\p R_\p)}{I_\p^{k}}.
\end{equation}

In order to prove Theorem \ref{Thm:AssInfty} in the monomial case, we recall a few basic facts from (Krull) dimension theory and Hilbert functions. For the convenience of the reader we provide a proof for each of these facts.
\begin{Remark}\label{Rem:elementary1}
	Let $R=\bigoplus_{k\ge0}R_k$ be a Noetherian graded $K$-algebra, $K$ a field, and let $M=\bigoplus_{k\ge0}M_k$ be a non-zero finitely generated graded $R$-module. Then, the Krull dimension $\dim M$ is positive if and only if $M_k\ne0$ for infinitely many $k$.
	\begin{proof}
		Recall that the Hilbert series of $M$ is $\textup{Hilb}_M(t)=\sum_{k\ge0}\dim_K(M_k)t^k$. By \cite[Theorem 6.1.3]{HH2011} there exists a Laurent polynomial $Q_M(t)\in\ZZ[t,t^{-1}]$ such that $\textup{Hilb}_M(t)=Q_M(t)/(1-t)^d$, where $d=\dim M$. It follows that $d=0$ if and only if $\textup{Hilb}_M(t)=Q_M(t)$. Since $Q_M(t)$ has only finitely many summands, this is the case if and only if $\dim_K M_k=0$ for all $k\gg0$, if and only if $M_k=0$ for all $k\gg0$. Consequently, $\dim M>0$ if and only if $M_k\ne 0$ for infinitely many $k$.
	\end{proof}
\end{Remark}

For a graded module $M=\bigoplus_{k\ge0}M_k$ we denote by $M_{\ge h}=\bigoplus_{k\ge h}M_k$ a truncation. We say that two graded modules $M=\bigoplus_{k\ge0}M_k$ and $N=\bigoplus_{k\ge0}N_k$ are equal up to truncation if there exist integers $h$ and $\ell$ such that $M_{\ge h}=N_{\ge\ell}$.
\begin{Remark}\label{Rem:elementary2}
	Let $R=\bigoplus_{k\ge0}R_k$ be a Noetherian graded $K$-algebra, $K$ a field, and let $M=\bigoplus_{k\ge0}M_k$ and $N=\bigoplus_{k\ge0}N_k$ be two finitely generated graded $R$-modules. If $M$ and $N$ are equal up to truncation then $\dim M=\dim N$.
	\begin{proof}
		Let $h,\ell>0$ such that $M_{\ge h}=N_{\ge\ell}$. Then $\dim M_{\ge h}=\dim N_{\ge\ell}$. Now, if $\dim {M_{\ge h}}=0$ then $M_k=N_k=0$ for all $k\gg0$. This implies $\dim N=\dim M=0$. Suppose now that $\dim M_{\ge h}=d>0$. We claim that $\dim M_{\ge h}=\dim M$ and that $\dim N_{\ge\ell}=\dim N$. By \cite[Theorem 6.1.3.(b)]{HH2011} the Hilbert function $\textup{Hilb}(M_{\ge h},k)=\dim_K ({M_{\ge h}})_k$ agrees with a polynomial $P(x)\in\QQ[x]$ of degree $d-1$ for all $k\gg0$. Since $\textup{Hilb}(M,k)=\textup{Hilb}(M_{\ge h},k)$ for all $k\ge h$, it follows that $\textup{Hilb}(M,k)$ agrees with $P(x)$ for all $k\gg0$. Now \cite[Theorem 6.1.3.(b)]{HH2011} implies $\dim M=d$, as well.
	\end{proof}
\end{Remark}

\begin{Remark}\label{Rem:elementary3}
	Let $R=\bigoplus_{k\ge0}R_k$ be a Noetherian graded $K$-algebra, $K$ a field, and let $I=\bigoplus_{k\ge0}I_k$ and $J=\bigoplus_{k\ge0}J_k$ be two graded ideals of $R$. Then $\dim IJ>0$ if and only if $\dim I>0$ and $J$ is non-zero or $\dim J>0$ and $I$ is non-zero.
	\begin{proof}
	Notice that $IJ=\bigoplus_{k\ge0}(IJ)_k$ is a graded ideal with $k$th graded component $(IJ)_k=\sum_{i+j=k}I_iJ_j$. Remark \ref{Rem:elementary1} implies that $\dim IJ>0$ if and only if for all $k\gg0$ there exist $i$ and $j$ such that $i+j=k$, $I_i\ne0$ and $J_{j}\ne0$. From this the assertion follows immediately.
	\end{proof}
\end{Remark}

We are now ready to give the promised proof.
\begin{proof}[Proof of Theorem \ref{Thm:AssInfty}, monomial case]
	Let $I$ and $J$ be monomial ideals. The associated primes of a monomial ideal are monomial prime ideals, that is, ideals generated by a subset of the variables \cite[Corollary 1.3.9]{HH2011}. Thus, to characterize the set $\Ass^\infty(I+J)$ we only need to consider monomial prime ideals containing $I+J$. Let $\p=(x_{i_1},\dots,x_{i_r},y_{j_1},\dots,y_{j_t})$ be such an ideal. To compute the Koszul homology module (\ref{eq:KoszHom}) we may replace ordinary localization with monomial localization. Indeed, recall that for a monomial ideal $L\subset S$, the \textit{monomial localization} of $L$ is the monomial ideal $L(\p)$ of the polynomial ring $S(\p)=K[x_{i_1},\dots,x_{i_r},y_{j_1},\dots,y_{j_t}]$ obtained by applying the substitutions $x_i\mapsto 1$ and $y_j\mapsto1$ for $x_i\notin\p$ and $y_j\notin \p$. Since we have the equality $L(\p)S_\p=L S_\p$, then for all $k\ge0$,
	$$
	\frac{((I+J)_\p^{k}:\p S_\p)}{(I+J)_\p^{k}}\ =\ \frac{((I+J)^{k}:\p)}{(I+J)^{k}}S_\p\ =\ \frac{((I(\p)+J(\p))^{k}:\p)}{(I(\p)+J(\p))^{k}}S_\p.
	$$
	Hence
	$$
	H_{\mu(\p S_\p)-1}(\p S_\p\,;\,\mathcal{R}(I+J)_\p)\ =\ H_{\mu(\p)-1}(\p\,;\,\mathcal{R}(I(\p)+J(\p))\,)S_\p.
	$$
	Since $S_\p$ is again a polynomial ring and $I(\p)$ and $J(\p)$ are monomial ideals of $S_\p$, we may assume from the very beginning that $\p=\m=(x_1,\dots,x_n,y_1,\dots,y_n)$ is the maximal ideal of $S$. We set $\m_x=(x_1,\dots,x_n)$ and $\m_y=(y_1,\dots,y_m)$.\smallskip
	
	Let $\gr_I(R)=\bigoplus_{k\ge0}(I^k/I^{k+1})$ and $\gr_J(T)=\bigoplus_{k\ge0}(J^k/J^{k+1})$ be the associated graded rings of $I\subset R$ and $J\subset T$. By \cite[Proposition 3.2]{HTT16} the canonical map
	$$
	\varphi\ :\ \gr_{I}(R)\otimes_K\gr_J(T)\rightarrow\gr_{I+J}(S)= \bigoplus_{k\ge0}(I+J)^k/(I+J)^{k+1}
	$$
	defined by $\varphi(f\otimes g)=fg$ is an isomorphism.
	
	Let $C(I)=(0:_{\gr_I(R)}\m_x)$, $C(J)=(0:_{\gr_J(T)}\m_y)$ and $C(I+J)=(0:_{\gr_{I+J}(S)}\m)$. We claim that
	\begin{equation}\label{eq:prodC}
		C(I+J)\ =\ C(I)C(J).
	\end{equation}
	First notice that $C(I)\otimes_KC(J)$ is an ideal of $\gr_{I}(R)\otimes_K\gr_J(T)=\gr_{I+J}(S)$ for it is the tensor product of two ideals. It is clear that
	$$
	C(I)C(J)\ =\ \varphi(C(I)\otimes_K C(J))\ \subseteq\ C(I+J).
	$$
	
	To show the opposite inclusion, using the subsequent calculation (\ref{eq:C(I+J)k}), we notice that $C(I+J)$, $C(I)$ and $C(J)$ are generated by residue classes of monomials. Hence $\varphi$ is a multigraded isomorphism. Let $u+(I+J)^{k+1}\in C(I+J)_k$ be non-zero, where $u\in(I+J)^k$ is a monomial. Then there exist integers $p$ and $q$ and monomials $u_1\in I^p$ and $u_2\in J^q$ such that $p+q=k$ and $\varphi((u_1+I^{p+1})\otimes(u_2+J^{q+1}))=u+(I+J)^{k+1}$. Let $1\le i\le n$. Since $x_i(u+I^{k+1})=0$ and $\alpha$ is an isomorphism, it follows that
	\[
	(x_iu_1+I^{p+1})\tensor (u_2+J^{q+1})\ =\ (x_i\tensor 1)((u_1+I^{p+1})\tensor (u_2+J^{q+1}))\ =\ 0.
	\]
	Notice that $u_2+J^{q+1}\neq 0$, otherwise $(u_1+I^{p+1})\tensor(u_2+J^{q+1})=0$, and then $u+(I+J)^{k+1}=0$, a contradiction. Hence $(x_iu_1+I^{p+1})=0$. This shows that $u_1+I^{p+1}\in C(I)$. Likewise $u_2+J^{q+1}\in C(J)$ and equation (\ref{eq:prodC}) follows.
	
	We claim that $C(I+J)$ and $H_{\mu(\m)-1}(\m\,;\,\mathcal{R}(I+J))$ are equal up to truncation. For this aim, we notice that for all $k\ge0$,
	\begin{equation}\label{eq:C(I+J)k}
		C(I+J)_k\ =\ \frac{((I+J)^{k+1}:\m)\cap(I+J)^{k}}{(I+J)^{k+1}}.
	\end{equation}
	(For this computation see also \cite[Theorem 2.2, proof]{FS2}). By \cite[Corollary 4.2]{R1976} we have $(I+J)^{k+1}:(I+J)=(I+J)^k$ for all $k\gg0$. Hence $((I+J)^{k+1}:\m)\subseteq(I+J)^{k+1}:(I+J)=(I+J)^k$ and so
	$$
	C(I+J)_k\ =\ \frac{((I+J)^{k+1}:\m)}{(I+J)^{k+1}}\ =\ H_{\mu(\m)-1}(\m\,;\,\mathcal{R}(I+J))_{k+1}.
	$$
	for all $k\gg0$. Similarly, $C(I)$ is equal up to truncation to $H_{\mu(\m_x)-1}(\m_x\,;\,\mathcal{R}(I))$ and $C(J)$ is equal up to truncation to $H_{\mu(\m_y)-1}(\m_y\,;\,\mathcal{R}(J))$. Remark \ref{Rem:elementary2} now implies that $\dim H_{\mu(\m)-1}(\m\,;\,\mathcal{R}(I+J))=\dim C(I+J)$. 
	
	Theorem \ref{Thm:pinAss8(I)} guarantees that $\mm\in\Ass^\infty(I+J)$ if and only if $\dim C(I+J)>0$. Equation (\ref{eq:prodC}) and Remark \ref{Rem:elementary3} imply that $\dim C(I+J)>0$ if and only if $\dim C(I)>0$ and $C(J)\ne0$ or $\dim C(J)>0$ and $C(I)\ne0$. Applying again Theorem \ref{Thm:pinAss8(I)} and Remark \ref{Rem:elementary1}, we see that $\dim C(I)>0$ if and only if $\m_x\in\Ass^\infty(I)$ and $\dim C(J)>0$ if and only if $\m_y\in\Ass^\infty(J)$. Whereas, since $\m_x$ is the maximal ideal of $R$, the condition $C(I)\ne0$ is equivalent to $\m_x\in\Ass^*(I)$. Likewise $C(J)\ne0$ if and only if $\m_y\in\Ass^*(J)$. From these observations the theorem follows.
\end{proof}

\section{The $\v$-function of sums}\label{Sec:4-A-P}

The purpose of this section is to prove the following very useful formula.
\begin{Theorem}\label{Thm:FormulaV-Num-MonId}
	Let $I\subset R$, $J\subset T$ be proper monomial ideals, and let $\p\in\Spec(R)$ and $\q\in\Spec(T)$ be monomial prime ideals such that $\p+\q\in\Ass((I+J)^k)$. Then,
	$$
	\v_{\p+\q}((I+J)^k)\ =\ \min_{\substack{0\le\ell<k\\ \p\in\Ass(I^{k-\ell})\\ \q\in\Ass(J^{\ell+1})}}(\v_\p(I^{k-\ell})+\v_\q(J^{\ell+1})).
	$$
\end{Theorem}

Before we prove the theorem we recall few basic facts about monomial ideals.\smallskip

Let $I\subset R=K[x_1,\dots,x_n]$ be a proper monomial ideal. As is customary, we denote by $\mathcal{G}(I)$ the minimal monomial generating set of $I$. Let $v\in R$ be a monomial. By \cite[Proposition 1.2.2]{HH2011}, a set of generators of $(I:v)$ is
\begin{equation}\label{eq:gensColon}
	\Big\{ \frac{\lcm(u,v)}{v}\ :\ u\in\mathcal{G}(I) \Big\}.
\end{equation}
From this fact it follows immediately that $(\sum_{j=1}^tI_t):v=\sum_{j=1}^t(I_j:v)$, for any proper monomial ideals $I_1,\dots,I_t\subset R$.

For vectors ${\bf a}=(a_1,\dots,a_n)\in\ZZ_{\ge0}^n$ and ${\bf b}=(b_1,\dots,b_m)\in\ZZ_{\ge0}^m$ we set ${\bf x^a}=\prod_ix_i^{a_i}$ and ${\bf y^b}=\prod_iy_i^{b_i}$. In particular, ${\bf x^0}={\bf y^0}=1$ for ${\bf 0}=(0,0,\dots,0)$.\medskip

Let $I\subset R$ and $J\subset T$ be proper monomial ideals. We claim that
\begin{equation}\label{eq:colonSepX-Y}
	(IJ:{\bf x^a}{\bf y^b})\ =\ (I:{\bf x^a})(J:{\bf y^b}).
\end{equation}
Indeed, $\mathcal{G}(IJ)=\{uv:u\in\mathcal{G}(I),v\in\mathcal{G}(J)\}$. Then, by equation (\ref{eq:gensColon}) any generator of the colon is of the form $\lcm(uv,{\bf x^ay^b})/({\bf x^ay^b})$, for $u\in\mathcal{G}(I)$ and $v\in\mathcal{G}(J)$. It is immediate to see that
$$
\frac{\lcm(uv,{\bf x^a}{\bf y^b})}{{\bf x^ay^b}}\ =\ \frac{\lcm(u,{\bf x^a})\lcm(v,{\bf y^b})}{{\bf x^ay^b}}\ =\ \frac{\lcm(u,{\bf x^a})}{{\bf x^a}}\cdot\frac{\lcm(v,{\bf y^b})}{{\bf y^b}}.
$$
From this fact equation (\ref{eq:colonSepX-Y}) follows immediately.

\begin{proof}[Proof of Theorem \ref{Thm:FormulaV-Num-MonId}]
	Let $\p\in\Spec(R)$ and $\q\in\Spec(T)$ be monomial prime ideals such that $\p+\q\in\Ass((I+J)^k)$. Up to a suitable relabeling, we may assume that $\p=(x_1,\dots,x_p)$ for some $1\le p\le n$ and $\q=(y_1,\dots,y_q)$ for some $1\le q\le m$.
	
	Let $f={\bf x^a}{\bf y^b}\in S$ be a monomial such that $(I+J)^k:f=\p+\q$ and with $\deg(f)=\v_{\p+\q}((I+J)^k)$. Note that
	\begin{equation}\label{eq:DecompColon}
		\begin{aligned}
			(I+J)^k:f\ &=\ (\sum_{h=0}^kI^{k-h}J^{h}):{\bf x^a}{\bf y^b}\ =\ \sum_{h=0}^k(I^{k-h}J^{h}:{\bf x^a}{\bf y^b})\\
			&=\ \phantom{(\!}\sum_{h=0}^k(I^{k-h}:{\bf x^a})(J^{h}:{\bf y^b}).
		\end{aligned}
	\end{equation}

	Therefore, for each $1\le i\le p$ there exists $0\le\ell_i\le k$ such that $x_if\in I^{k-\ell_i}J^{\ell_i}$. Let $\ell=\max\{\ell_1,\dots,\ell_p\}$. Then $x_i{\bf x^a}\in I^{k-\ell}$ for all $1\le i\le p$ and ${\bf y^b}\in J^\ell$. Thus $x_if\in I^{k-\ell}J^{\ell}$ for all $1\le i\le p$. Notice that $\ell<k$. Otherwise if $\ell=k$, since ${\bf y^b}\in J^k$ then $f={\bf x^ay^b}\in J^k\subset(I+J)^k$ and $(I+J)^k:f=S$ against the assumption.
	
	Since ${\bf y^b}\in J^\ell$, the $\ell$th summand $(I^{k-\ell}:{\bf x^a})(J^\ell:{\bf y^b})$ in (\ref{eq:DecompColon}) is equal to $(I^{k-\ell}:{\bf x^a})$ and by our discussion contains $\p$. Notice that $(I^{k-\ell}:{\bf x^a})$ is a monomial ideal of $R$. Therefore, $(I^{k-\ell}:{\bf x^a})$ must be equal to $\p$, for otherwise $(I+J)^k:f$ would strictly contain $\p+\q$, against the assumption. Hence $(I^{k-\ell}:{\bf x^a})=\p$ and so
	\begin{equation}\label{eq:IneqV-Num1}
		\deg({\bf x^a})\ \ge\ \v_\p(I^{k-\ell}).
	\end{equation}

    Since $x_1{\bf x^a}\in I^{k-\ell}$, there exist $k-\ell$ monomials $u_1,u_2\dots,u_{k-\ell}\in I$ such that $x_1{\bf x^a}=u_1\cdots u_{k-\ell}$. Consequently, $x_1$ divides $u_j$ for some $1\le j\le k-\ell$. Up to relabeling we may assume $j=k-\ell$. We deduce that $u_1\cdots u_{k-\ell-1}$ divides ${\bf x^a}$. Hence ${\bf x^a}\in I^{k-(\ell+1)}$ and the $(\ell+1)$th summand in (\ref{eq:DecompColon}) becomes
    $$
    (I^{k-(\ell+1)}:{\bf x^a})(J^{\ell+1}:{\bf y^b})\ =\ (J^{\ell+1}:{\bf y^b}).
    $$
    Since ${\bf x^a}\in I^{k-(\ell+1)}$, then $(I^{k-h}:{\bf x^a})=S$ for all $\ell+1\le h\le k$. Similarly, since ${\bf y^b}\in J^\ell$, then $(J^h:{\bf y^b})=S$ for all $1\le h\le\ell$. Thus, equation (\ref{eq:DecompColon}) simplifies to
    $$
    (I+J)^k:f\ =\ \sum_{h=0}^{\ell}(I^{k-h}:{\bf x^a})+\sum_{h=\ell+1}^k(J^h:{\bf y^b}).
    $$
    
    Noticing that $(I^k:{\bf x^a})\subseteq\cdots\subseteq(I^{k-\ell}:{\bf x^a})$ and $(J^{\ell+1}:{\bf y^b})\supseteq\cdots\supseteq(J^k:{\bf y^b})$, the previous formula simplifies to
    $$
    (I+J)^k:{\bf x^ay^b}\ =\ (I^{k-\ell}:{\bf x^a})+(J^{\ell+1}:{\bf y^b}).
    $$
    
    Since $(I+J)^k:{\bf x^ay^b}=\p+\q$ by assumption and $(I^{k-\ell}:{\bf x^a})=\p$ we conclude that $(J^{\ell+1}:{\bf y^b})=\q$. Hence
    \begin{equation}\label{eq:IneqV-Num2}
    	\deg({\bf y^b})\ \ge\ \v_\q(J^{\ell+1}).
    \end{equation}

    Putting together (\ref{eq:IneqV-Num1}) and (\ref{eq:IneqV-Num2}), we deduce that
    $$
    \v_{\p+\q}((I+J)^k)\ =\ \deg(f)\ =\ \deg({\bf x^a})+\deg({\bf y^b})\ \ge\ \v_\p(I^{k-\ell})+\v_\q(J^{\ell+1}),
    $$
    where $0\le\ell<k$. Hence
    $$
    \v_{\p+\q}((I+J)^k)\ \ge\ \min_{\substack{0\le\ell<k\\ \p\in\Ass(I^{k-\ell})\\ \q\in\Ass(J^{\ell+1})}}(\v_\p(I^{k-\ell})+\v_\q(J^{\ell+1})).
    $$
    
    To show the opposite inequality, let $0\le\ell<k$ such that $\p\in\Ass(I^{k-\ell})$ and $\q\in\Ass(J^{\ell+1})$. Let ${\bf x^a},{\bf y^b}\in S$ be monomials such that $(I^{k-\ell}:{\bf x^a})=\p$, $(J^{\ell+1}:{\bf y^b})=\q$, $\deg({\bf x^a})=\v_\p(I^{k-\ell})$ and $\deg({\bf y^b})=\v_\q(J^{\ell+1})$.
    
    Set $f={\bf x^a}{\bf y^b}$. We claim that $(I+J)^k:f=\p+\q$. This implies the desired inequality and concludes the proof.
    
    Arguing as in the previous part of the proof, we see that ${\bf x^a}\in I^{k-(\ell+1)}$ and ${\bf y^b}\in J^\ell$. Hence $(I^{k-h}:{\bf x^a})=S$ for all $\ell+1\le h\le k$, and $(J^h:{\bf y^b})=S$ for all $1\le h\le\ell$. Then equation (\ref{eq:DecompColon}) which is valid in our setting, simplifies to
    $$
    (I+J)^k:f\ =\ \sum_{h=0}^{\ell}(I^{k-h}:{\bf x^a})+\sum_{h=\ell+1}^k(J^h:{\bf y^b}).
    $$
    Since $(I^k:{\bf x^a})\subseteq\cdots\subseteq(I^{k-\ell}:{\bf x^a})=\p$ and $\q=(J^{\ell+1}:{\bf y^b})\supseteq\cdots\supseteq(J^k:{\bf y^b})$, the previous equation becomes $(I+J)^k:f=\p+\q$, as desired.
\end{proof}

\begin{Corollary}
	Let $I\subset R$, $J\subset T$ be proper monomial ideals. Then,
	$$
	\v((I+J)^k)\ =\ \min_{\substack{\p\in\Spec(R)\\\q\in\Spec(T)\\ \p+\q\in\Ass((I+J)^k)}}(\min_{\substack{0\le\ell<k\\ \p\in\Ass(I^{k-\ell})\\ \q\in\Ass(J^{\ell+1})}}(\v_\p(I^{k-\ell})+\v_\q(J^{\ell+1}))).
	$$
\end{Corollary}

Due to some experimental evidence, we could expect that the formula given in Theorem \ref{Thm:FormulaV-Num-MonId} holds true for any graded ideals $I\subset R$ and $J\subset T$.\medskip

In \cite{FSPackA} the following question was posed.
\begin{Question}\label{Ques:v-reg}
	\textup{\cite[Question 5.1]{FSPackA}} Let $I\subset S$ be an homogeneous ideal. Is it true that the inequality $\v(I^k)<\reg(I^k)$ holds for all $k\gg0$?
\end{Question}
Some partial results supporting this expectation are provided in \cite{BMS24,FS2,VS24}.\smallskip

The \textit{support} of a monomial $u\in S$ is defined as the set $\supp(u)=\{x_i:x_i\ \textup{divides}\ u\}$. For a proper monomial ideal $I\subset S$, we define the \textit{support} of $I$ to be the set $\supp(I)=\bigcup_{u\in\mathcal{G}(I)}\supp(u)$.

We say that a monomial ideal $I\subset S$ is \textit{connected} if for any $u,v\in\mathcal{G}(I)$ there exist $w_0,w_1,\dots,w_r\in \mathcal{G}(I)$ such that $w_0=u$, $w_r=v$ and $\supp(w_i)\cap\supp(w_{i+1})\ne\emptyset$ for all $0\le i\le r-1$. It is clear that any monomial ideal $I\subset S$ can be written uniquely as $I=I_1+\dots+I_t$, where each $I_j$ is a connected monomial ideal, and $\mathcal{G}(I)$ is the disjoint union $\bigcup_{j=1}^t\mathcal{G}(I_j)$. Notice in particular that $I_1,\dots,I_t$ have pairwise disjoint supports.

In the next corollary, we show that to prove Question \ref{Ques:v-reg} for all monomial ideals, it is enough to consider only connected monomial ideals.

\begin{Corollary}\label{Cor:regIneq}
	Let $I\subset R$ and $J\subset T$ be proper monomial ideals. Suppose that $\v(I^k)<\reg(I^k)$ and $\v(J^k)<\reg(J^k)$ for all $k\gg0$. Then $\v((I+J)^k)<\reg((I+J)^k)$ for all $k\gg0$.
\end{Corollary}
\begin{proof}
	Applying \cite[Theorem 3.3]{HTT16} for all $k\ge1$ we have
	$$
	\reg\frac{(I+J)^{k-1}}{(I+J)^k}\ =\ \max_{i+j=k-1}\big\{\!\reg\frac{I^{i}}{I^{i+1}}+\reg\frac{J^{j}}{J^{j+1}}\big\}.
	$$
	Setting $j=\ell$, we have $0\le\ell<k$ and $i+1=k-\ell$, and the formula becomes
	$$
	\reg\frac{(I+J)^{k-1}}{(I+J)^k}\ =\ \max_{0\le\ell<k}\big\{\!\reg\frac{I^{k-\ell-1}}{I^{k-\ell}}+\reg\frac{J^{\ell}}{J^{\ell+1}}\big\}.
	$$
	By \cite[Lemma 5.1]{HTT16}, we have $\reg L^{p-1}/L^{p}=\reg L^{p}-1$, for any homogeneous ideal $L\subset S$ and all $p\gg0$. Thus, for any $k\gg0$ we can find $0\le\ell<k$ such that $k-\ell$ is also large enough and such that
	\begin{enumerate}
		\item[(a)] $\reg((I+J)^k)\ge\reg(I^{k-\ell})+\reg(J^{\ell+1})-1$,
		\item[(b)] $\v(I^{k-\ell})<\reg(I^{k-\ell})$, and
		\item[(c)] $\v(J^{\ell+1})<\reg(J^{\ell+1})$.
	\end{enumerate}
    There exist $\p\in\Ass(I^{k-\ell})$ and $\q\in\Ass(J^{\ell+1})$ such that $\v(I^{k-\ell})=\v_\p(I^{k-\ell})$ and $\v(J^{\ell+1})=\v_\q(J^{\ell+1})$. Applying Theorem \ref{Thm:FormulaV-Num-MonId} and (a), (b) and (c) we conclude that
	\begin{align*}
		\reg((I+J)^k)\ &\ge\ \reg(I^{k-\ell})+\reg(J^{\ell+1})-1\ >\ \v(I^{k-\ell})+\reg(J^{\ell+1})-1\\
		&\ge\ \v(I^{k-\ell})+\v(J^{\ell+1})\ =\ \v_\p(I^{k-\ell})+\v_\q(J^{\ell+1})\\
		&\ge\ \v_{\p+\q}((I+J)^k)\ \ge\ \v((I+J)^k),
	\end{align*}
    as desired.
\end{proof}

\section{Consequences}\label{Sec:5-A-P}

In this section we record several consequences of Theorem \ref{Thm:FormulaV-Num-MonId}.

Let $I\subset S$ be a graded ideal and let $\p\in\Ass^\infty(I)$. Let $f(k)=ak+b$ be the (unique) linear function such that $\v(I^k)=f(k)$ for all $k\gg0$. The \textit{$\v$-stability index} of $I$, denoted by $\vstab(I)$, is defined as the smallest $k_0$ such that $\v(I^k)=f(k)$ for all $k\ge k_0$. Likewise, let $g(k)=ck+d$ be the (unique) linear function such that $\v_\p(I^k)=g(k)$ for all $k\gg0$. The \textit{$\v_\p$-stability index} of $I$, denoted by $\v_\p\text{-stab}(I)$, is defined as the smallest $k_0$ such that $\v_\p(I^k)=g(k)$ for all $k\ge k_0$.

\begin{Example}\label{ex:Pedro}
	Let ${f \in S}$ be a homogeneous polynomial and consider the principal ideal ${I=(f)}$. If ${f=f_1^{\,a_1} \cdots f_s^{\,a_s}}$ is the decomposition of $f$ into irreducible polynomials, with ${\deg f_1 \le \cdots \le \deg f_s}$, then ${\Ass^\infty(I)=\Ass(I)=\{(f_1), \ldots, (f_s)\}}$, and we have that for any ${k>0}$, and any ${r \in \{1,\ldots,s\}}$, ${(I^k:(f^k/f_r))=(f_r)}$. So the $\v$-number of $I^k$ is attained by ${(I^k:(f^k/f_s))=(f_s)}$, and we have
	\[
	\v(I^k)=k \deg f - \deg f_s = \alpha(I)k- \deg f_s.
	\]
	In particular, the $\v$-stability index of $I$ is $1$.
\end{Example}

\begin{Theorem}\label{Thm:DisjointSupports}
	Let $I_1,\dots,I_t\subset S=K[x_1,\dots,x_n]$ be monomial ideals with pairwise disjoint supports. Suppose that $\v(I_j^k)=\alpha(I_j)k-1$ for all $k\gg0$ and all $j=1,\dots,t$. Then, for all $k\ge1$,
	$$
	\v((\sum_{j=1}^tI_j)^k)\ \ge\ (\min_{1\le j\le t}\alpha(I_j))k+(\sum_{j=1}^t\alpha(I_j)-\min_{1\le j\le t}\alpha(I_j)-t).
	$$
	Equality holds for all $k\gg0$ if $\alpha(I_1)=\alpha(I_2)=\cdots=\alpha(I_t)$ or for all $k\ge1$ if $\vstab(I_1)=\vstab(I_2)=\cdots=\vstab(I_t)=1$.
\end{Theorem}
\begin{proof}
	Up to relabeling, we may assume that $\alpha(I_1)\le\alpha(I_2)\le\cdots\le\alpha(I_t)$. Proceeding by induction on $t\ge1$ we show that
	\begin{equation}\label{eq:ineq-v-sum}
		\v((\sum_{j=1}^tI_j)^k)\ \ge\ \alpha(I_1)k+(\sum_{j=2}^t\alpha(I_j)-t)
	\end{equation}
	for all $k\ge1$, and equality holds for all $k\gg0$ if $\alpha(I_1)=\alpha(I_2)=\dots=\alpha(I_t)$ or for all $k\ge1$ if $\vstab(I_1)=\vstab(I_2)=\cdots=\vstab(I_t)=1$.
	
	For $t=1$, the inequality holds by \cite[Proposition 2.2]{F2023} and the equality holds for all $k\gg0$ by assumption, and for all $k\ge1$ if $\vstab(I_1)=1$ by definition of $\v$-stability.
	
	Let $t>1$. We set $I=I_1+\dots+I_{t-1}$ and $J=I_t$. By inductive hypothesis
	\begin{equation}\label{eq:ineqIndHyp}
		\v(I^k)\ =\ \v((\sum_{j=1}^{t-1}I_j)^k)\ \ge\ \alpha(I_1)k+(\sum_{j=2}^{t-1}\alpha(I_j)-(t-1))
	\end{equation}
	for all $k\ge1$.
	
	Since $\v((I+J)^k)=\min_{P\in\Ass((I+J)^k)}\v_P((I+J)^k)$, to prove the inequality (\ref{eq:ineq-v-sum}), it is enough to show that $\v_P((I+J)^k)$ satisfies the corresponding inequality for all $P\in\Ass((I+J)^k)$ and all $k\ge1$.
	
	Let $P\in\Ass((I+J)^k)$. Then \cite[Corollary 1.3.9]{HH2011} implies that $P=\p+\q$ where $\p\in\Spec(R)$ and $\q\in\Spec(T)$ are monomial prime ideals. By Theorem \ref{Thm:FormulaV-Num-MonId},
	$$
	\v_{P}((I+J)^k)\ =\ \min_{\substack{0\le\ell<k\\ \p\in\Ass(I^{k-\ell})\\ \q\in\Ass(J^{\ell+1})}}(\v_\p(I^{k-\ell})+\v_\q(J^{\ell+1})).
	$$
	Now, let $0\le\ell<k$ such that $\v_{P}((I+J)^k)=\v_\p(I^{k-\ell})+\v_\q(J^{\ell+1})$. Then, using (\ref{eq:ineqIndHyp}) and since by \cite[Proposition 2.2]{F2023} we have $\v_\q(J^{\ell+1})\ge\alpha(J)(\ell+1)-1=\alpha(I_t)(\ell+1)-1$, we obtain that
	\begin{align*}
		\v_{P}((I+J)^k)\ &=\ \v_\p(I^{k-\ell})+\v_\q(J^{\ell+1})\ \ge\ \v(I^{k-\ell})+\v_\q(J^{\ell+1})\\[10pt]
		&\ge\ \alpha(I_1)(k-\ell)+(\sum_{j=2}^{t-1}\alpha(I_j)-(t-1))+\alpha(I_t)(\ell+1)-1\\
		&\ge\ \alpha(I_1)(k-\ell)+(\sum_{j=2}^{t-1}\alpha(I_j)-(t-1))+\alpha(I_1)\ell+\alpha(I_t)-1\\
		&=\ \alpha(I_1)k+(\sum_{j=2}^{t}\alpha(I_j)-t).
	\end{align*}
    Consequently, inequality (\ref{eq:ineq-v-sum}) holds.
    
    Suppose now that $\alpha(I_1)=\alpha(I_2)=\dots=\alpha(I_t)=a$. To prove that equality holds in (\ref{eq:ineq-v-sum}) for all $k\gg0$, let $\p\in\Ass^\infty(I)$, $\q\in\Ass^\infty(J)$ such that $\v_\p(I^k)=\v(I^k)$ and $\v_\q(J^k)=\v(J^k)$ for all $k\gg0$. Set $P=\p+\q$. Then $P\in\Ass^\infty(I+J)$ by Theorem \ref{Thm:AssInfty}. For all $k\gg0$ we can find $0\le\ell<k$ big enough such that $\p\in\Ass(I^{k-\ell})$, $\v_\p(I^{k-\ell})=\v(I^{k-\ell})=a(k-\ell)+((t-2)a-(t-1))$, $\q\in\Ass(J^{\ell+1})$ and $\v_\q(J^{\ell+1})=\v(J^{\ell+1})=a(\ell+1)-1$. Again by Theorem \ref{Thm:FormulaV-Num-MonId} we have
    \begin{align*}
    	\v((I+J)^k)\ &\le\ \v_P((I+J)^k)\ \le\ \v_\p(I^{k-\ell})+\v_\q(J^{\ell+1})\\
    	&=\ a(k-\ell)+((t-2)a-(t-1))+a(\ell+1)-1\\
    	&=\ ak+((t-1)a-t).
    \end{align*}
    This last inequality together with inequality (\ref{eq:ineq-v-sum}) implies the assertion.
    
    Finally, suppose that $\vstab(I_1)=\vstab(I_2)=\cdots=\vstab(I_t)=1$. To prove that equality holds in (\ref{eq:ineq-v-sum}) for all $k\ge1$, let $\p\in\Ass^\infty(I)$ such that $\v_\p(I^{k})=\v(I^{k})$ and let $\q\in\Ass(J)$ such that $\v_\q(J)=\v(J)$. Set $P=\p+\q$. Then $P\in\Ass^\infty(I+J)$ by Theorem \ref{Thm:AssInfty}. By induction on $t$, $\v_\p(I^k)=\v(I^k)=\alpha(I_1)k+(\sum_{j=2}^{t-1}\alpha(I_j)-(t-1))$ for all $k\ge1$ and since $\vstab(I_t)=1$, $\v_\q(J)=\v(J)=\alpha(J)-1=\alpha(I_t)-1$. Again by Theorem \ref{Thm:FormulaV-Num-MonId} we have
    \begin{align*}
    	\v((I+J)^k)\ &\le\ \v_P((I+J)^k)\ \le\ \v_\p(I^{k})+\v_\q(J)\\[10.5pt]
    	&=\ \alpha(I_1)(k-\ell)+(\sum_{j=2}^{t-1}\alpha(I_j)-(t-1))+\alpha(I_t)-1\\
    	&=\ \alpha(I_1)(k-\ell)+(\sum_{j=2}^{t}\alpha(I_j)-t).
    \end{align*}
    This last inequality together with inequality (\ref{eq:ineq-v-sum}) implies the assertion.
\end{proof}

Notice that for any graded ideal $I\subset S$, we have $\reg(I^k)\ge\alpha(I)k$ for all $k\ge1$. Combining this observation with Theorem \ref{Thm:DisjointSupports} and Corollary \ref{Cor:regIneq} we obtain
\begin{Corollary}\label{Cor:regBound}
	Let $I_1,\dots,I_t\subset S=K[x_1,\dots,x_n]$ be monomial ideals with pairwise disjoint supports. Suppose that $\v(I_j^k)=\alpha(I_j)k-1$ for all $k\gg0$ and all $j=1,\dots,t$. Then, for all $k\gg0$
	$$
	\reg((\sum_{j=1}^tI_j)^k)\ >\ \v((\sum_{j=1}^tI_j)^k)\ \ge\ (\min_{1\le j\le t}\alpha(I_j))k+(\sum_{j=1}^t\alpha(I_j)-\min_{1\le j\le t}\alpha(I_j)-t).
	$$
\end{Corollary}

Let $G$ be a finite simple graph on the vertex set $V(G)=\{x_1,\dots,x_n\}$ and edge set $E(G)$. The \textit{edge ideal} of $G$ is the monomial ideal $I(G)$ of $S=K[x_1,\dots,x_n]$ generated by the monomials $x_ix_j$ such that $\{x_i,x_j\}\in E(G)$.

We denote by $c(G)$ the number of connected components of $G$.
\begin{Corollary}\label{Cor:v(I(G)^k)}
	Let $G$ be a finite simple graph. Then, for all $k\gg0$,
	$$
	\v(I(G)^k)\ =\ 2k+(c(G)-2).
	$$
\end{Corollary}
\begin{proof}
	Let $G_1,\dots,G_{c(G)}$ be the connected components of $G$. Set $I_i=I(G_i)$ for $i=1,\dots,c(G)$. Then $I(G)=I_1+\dots+I_{c(G)}$. It is proved in \cite[Theorem 5.1]{BMS24} that if $H$ is a connected graph, then $\v(I(H)^k)=2k-1$ for all $k\gg0$. Hence, we have $\v(I_i^k)=2k-1$ for all $i=1,\dots,c(G)$ and all $k\gg0$. Since $I_1,\dots,I_{c(G)}$ have pairwise disjoint supports, Theorem \ref{Thm:DisjointSupports} implies that
	\begin{align*}
		\v(I(G)^k)\ &=\ (\min_{1\le j\le c(G)}\alpha(I_j))k+(\sum_{j=1}^{c(G)}\alpha(I_j)-\min_{1\le j\le c(G)}\alpha(I_j)-t)\\
		&=\ 2k+(2c(G)-2-c(G))\\
		&=\ 2k+(c(G)-2),
	\end{align*}
    as desired.
\end{proof}

\begin{Example}\label{ex:cycle}
	\rm Let $G=C_5$ be the 5-cycle on vertex set $\{x_1,x_2,x_3,x_4,x_5\}$. That is $I(G)=(x_1x_2,x_2x_3,x_3x_4,x_4x_5,x_5x_1)$. By \cite[Proposition 5.11]{BMS24}, we have that $\v(I(G)^k)=2k-1$ for all $k\ge2=\lfloor\frac{5}{2}\rfloor$. Using the package \texttt{VNumber} \cite{FSPack}, we checked that $\v(I(G))=2$. Thus $\vstab(I(G))=2$.
\end{Example}

The conclusion of Theorem \ref{Thm:DisjointSupports} is no longer valid if the ideals $I_1,\dots,I_t$ do not have the same initial degree, or if $\vstab(I_j)>1$ for some $j$. We demonstrate this with the next example.
\begin{Example}
	\rm Let $S=K[x_1,\dots,x_5,y_1,\dots,y_5,z_1,\dots,z_5]$, and let
	\begin{align*}
		I\ &=\ (x_1x_2,x_2x_3,x_3x_4,x_4x_5,x_5x_1),\\
		J\ &=\ (y_1y_2,\,y_2y_3,\,y_3y_4,\,y_4y_5,\,y_5y_1),\\
		L\ &=\ (z_1z_2,\,z_2z_3,\,z_3z_4,\,z_4z_5,\,z_5z_1).
	\end{align*}
    We set $I_1=I$ and $I_2=JL$. Then $I_1$ and $I_2$ have disjoint supports. Notice that $I,J,L$ are the edge ideals of the 5-cycles on vertex sets $\{x_1,\dots,x_5\}$, $\{y_1,\dots,y_5\}$, $\{z_1,\dots,z_5\}$, respectively. By Example \ref{ex:cycle}, we have
    $$
    \v(I^k)\ =\ \v(J^k)\ =\ \v(L^k)\ =\ \begin{cases}
    	\hfill2&\textup{if}\ k=1,\\ 2k-1&\textup{otherwise.}
    \end{cases}
    $$
    By Corollary \ref{Cor:Prod}, we obtain that
    $$
    \v(I_2^k)\ =\ \begin{cases}
    	\hfill4&\textup{if}\ k=1,\\ 4k-1&\textup{otherwise.}
    \end{cases}
    $$
    Thus $I_1$ and $I_2$ satisfy the equation $\v(I_i^k)=\alpha(I_i)k-1$ for all $k\ge2$. Notice that $\alpha(I_1)=2\ne4=\alpha(I_2)$ and $\vstab(I_1)=\vstab(I_2)=2>1$. By Theorem \ref{Thm:DisjointSupports} we have that $\v((I_1+I_2)^k)\ \ge\ 2k+2$ for all $k\gg0$. We now show that
    $$
    \v((I_1+I_2)^k)\ =\ 2k+3
    $$
    for all $k\gg0$.\\
	For this aim, we apply Theorem \ref{Thm:FormulaV-Num-MonId}. Let $k\gg0$ and let $P\in\Ass^\infty(I_1+I_2)$. There exist $0\le\ell<k$, $\p\in\Ass(I_1^{k-\ell})$ and $\q\in\Ass(I_2^{\ell+1})$ such that $P=\p+\q$. We have
	$$
	\v_\p(I_1^{k-\ell})+\v_\q(I_2^{\ell+1})\ \ge\ \v(I_1^{k-\ell})+\v(I_2^{\ell+1})\ =\ \begin{cases}
		\hfill 2k+3&\textup{if}\ \ell=0,\\
		2k+(2\ell+2)&\textup{if}\ 0<\ell<k.
	\end{cases}
	$$
	By Theorem \ref{Thm:DisjointSupports} we then deduce that $\v_P((I_1+I_2)^k)\ge 2k+3$ for all $k\gg0$. Since $P$ is arbitrary, then $\v((I_1+I_2)^k)\ge 2k+3$ for all $k\gg0$. By choosing $\ell=0$, $\p\in\Ass(I_1^k)$ such that $\v_\p(I_1^k)=\v(I_1^k)=2k-1$ and $\q\in\Ass(I_2)$ such that $\v_\q(I_2)=\v(I_2)=4$, we see that the desired equality indeed holds for all $k\gg0$.
\end{Example}

The following result follows immediately by Corollaries \ref{Cor:regBound} and \ref{Cor:v(I(G)^k)}. However, we provide a different proof using a result of Beyarslan, H\`a and Trung \cite[Theorem 4.5]{BHT15}.
\begin{Corollary}
	Let $G$ be a finite simple graph. Then, for all $k\gg0$,
	$$
	\v(I(G)^k)\ <\ \reg(I(G)^k).
	$$
\end{Corollary}
\begin{proof}
	Let $M=\{e_1,\dots,e_r\}$ be a subset of $E(G)$. Recall that $M$ is a called a $r$-matching of $G$ if $e_i\cap e_j=\emptyset$ for all $i\ne j$. Whereas, $M$ is called an induced $r$-matching if $M$ is a $r$-matching and the only edges of the induced subgraph of $G$ on the vertex set $V(M)=\bigcup_{i=1}^re_i$ are the edges of $M$. We call $r$ the size of $M$. We denote by $\textup{im}(G)$ the largest size of an induced matching of $G$. It is clear that $\textup{im}(G)\ge c(G)$. Indeed, let $G_1,\dots,G_{c(G)}$ be the connected components of $G$, and pick an edge $e_i\in E(G_i)$ for each $i=1,\dots,c(G)$. Then $M=\{e_1,\dots,e_{c(G)}\}$ is an induced $c(G)$-matching of $G$.
	
	It is shown in \cite[Theorem 4.5]{BHT15} that $\reg(I(G)^k)\ge 2k+\textup{im}(G)-1$ for all $k\ge1$. Since $\textup{im}(G)\ge c(G)$ and by Corollary \ref{Cor:v(I(G)^k)} we have $\v(I(G)^k)=2k+(c(G)-2)$ for all $k\gg0$, the assertion follows.
\end{proof}

Recall that an ideal $I\subset S$ is called a \textit{complete intersection} if $I$ is generated by a regular sequence. It is well-known that a monomial ideal $I\subset S$ is a complete intersection if and only if $\supp(u)\cap\supp(v)=\emptyset$ for all $u,v\in\mathcal{G}(I)$ with $u\ne v$.
\begin{Corollary}\label{Cor:Pedro}
	Let $I\subset S=K[x_1,\dots,x_n]$ be a monomial complete intersection.\smallskip
	\begin{enumerate}
		\item[\textup{(a)}] We have
		$$
		\Ass^\infty(I)\ =\ \{\!\!\sum_{u\in\mathcal{G}(I)}(x_{u})\ :\ x_{u}\in\supp(u)\}.
		$$
		\item[\textup{(b)}] $\Ass(I^k)=\Ass^\infty(I)$ for all $k\ge1$.
		\item[\textup{(c)}] For all $k\ge1$,
		$$
		\v(I^k)\ =\ \alpha(I)k+(\!\!\sum_{u\in\mathcal{G}(I)}\deg(u)-\alpha(I)-\mu(I)).
		$$
	\end{enumerate} 
\end{Corollary}
\begin{proof}
	Let $u=x_{i_1}^{a_{i_1}}\cdots x_{i_d}^{a_{i_d}}\in S$ be a monomial with $i_1<i_2<\dots<i_d$ and $a_{i_j}\ge1$ for all $1\le j\le d$. Then $(x_{i_1}^{a_{i_1}})\cap\dots\cap(x_{i_d}^{a_{i_d}})$ is the primary decomposition of $I=(u)$, and $\Ass(I)=\{(x_i):x_i\in\supp(u)\}$. Notice that $\supp(u)=\supp(u^k)$ for all $k\ge1$. Hence $\Ass(I^k)=\Ass^\infty(I)$ for all $k\ge1$. Furthermore, for each $(x_{i_j})\in\Ass(I^k)$, we have $(I^k:(u/x_{i_j}))=(x_{i_j})$ and so $\v(I^k)\le\deg(u^k/x_{i_j})=\deg(u^k)-1=\alpha(I)k-1$ for all $k\ge1$. By \cite[Proposition 2.2]{F2023} we also have $\v(I^k)\ge\alpha(I)k-1$ for all $k\ge1$. Hence $\v(I^k)=\alpha(I)k-1$ for all $k\ge1$. Thus (a), (b) and (c) hold if $\mu(I)=1$.
	
	Suppose now that $\mu(I)>1$. Let $\mathcal{G}(I)=\{u_1,\dots,u_m\}$. Set $J=(u_1,\dots,u_{m-1})$ and $L=(u_m)$. Then $I=J+L$ and $J$ and $L$ have pairwise disjoint supports. By induction on $\mu(I)$, the ideals $J$ and $L$ satisfy statements (a), (b) and (c). By \cite[Theorem 4.1 (3)]{NguyenT} we have for all $k\ge1$
	$$
	\Ass(I^k)\ =\ \bigcup_{\substack{0\le\ell<k\\ \p\in\Ass(J^{k-\ell})\\ \q\in\Ass(L^{\ell+1})}}\{\p+\q\}.
	$$
	Since $\Ass(J^k)=\Ass^\infty(J)$ and $\Ass(L^k)=\Ass^\infty(L)$ for all $k\ge1$, the previous equation implies statements (a) and (b). Statement (c) follows by Theorem \ref{Thm:DisjointSupports} since $\v((u_i)^k)=\deg(u_i)k-1$ and $\vstab((u_i))=1$ for all $1\le i\le m$.
\end{proof}

We end the paper with another application of Theorem \ref{Thm:DisjointSupports}.\smallskip

Following \cite{MKA16}, (see also \cite{CF3,M2023}) we say that a monomial ideal $I\subset S=K[x_1,\dots,x_n]$ is \textit{vertex splittable} if it can be obtained by the following recursive procedure.
\begin{enumerate}
	\item[(i)] If $u$ is a monomial and $I=(u)$, $I=0$ or $I=S$, then $I$ is vertex splittable.\smallskip
	\item[(ii)] If there exists a variable $x_i$ and vertex splittable ideals $I_1\subset S$ and $I_2\subset K[x_1,\dots,x_{i-1},x_{i+1},\dots,x_n]$ such that $I=x_iI_1+I_2$, $I_2\subseteq I_1$ and $\mathcal{G}(I)$ is the disjoint union of $\mathcal{G}(x_iI_1)$ and $\mathcal{G}(I_2)$, then $I$ is vertex splittable.
\end{enumerate}
In the case (ii), the decomposition $I=x_iI_1+I_2$ is called a \textit{vertex splitting} of $I$ and $x_i$ is called a \textit{splitting vertex} of $I$.

\begin{Examples}
	\rm Many classes of monomial ideals are vertex splittable.\medskip
	
	(a) Edge ideals with linear resolution are vertex splittable, see \cite[Theorem 3.7 and Corollary 3.8]{MKA16}.\medskip
	
	(b) Cover ideals of Cohen-Macaulay very well-covered graphs, which include the Hibi ideals, are vertex splittable, see the proofs of \cite[Lemma 3.4 and Corollary 3.5]{CF2} or \cite[Proposition 2.3 and Lemma 3.3]{CF1}.\medskip
	
	(c) Let ${\bf t}=(t_1\dots,t_{d-1})\in\ZZ_{\ge0}^{d-1}$ with $d\ge2$. Following \cite{F-Vect,CF}, we say that $u=x_{i_1}\cdots x_{i_\ell}\in S$, $i_1\le\cdots\le i_\ell$, is a ${\bf t}$-spread monomial if $\ell\le d$ and $i_{j+1}-i_j\ge t_j$ for $j=1,\dots,\ell-1$. A monomial ideal $I\subset S$ is called ${\bf t}$-spread if each $u\in\mathcal{G}(I)$ is ${\bf t}$-spread, and it is called ${\bf t}$-spread strongly stable if $I$ is ${\bf t}$-spread and for all $u\in\mathcal{G}(I)$ and all $i<j$ such that $x_j\in\supp(u)$ and $x_i(u/x_j)$ is ${\bf t}$-spread, then $x_i(u/x_j)\in I$. It is shown in \cite[Proposition 1(a)]{CF3} that ${\bf t}$-spread strongly stable ideals are vertex splittable.\medskip
	
	(d) For a monomial $u\in S$, we set $\deg_{x_i}(u)=\max\{j:x_i^j\ \textup{divides}\ u\}$. A monomial ideal $I\subset S$ generated in a single degree is called polymatroidal if the exchange property holds: for all $u,v\in\mathcal{G}(I)$ and all $i$ with $\deg_{x_i}(u)>\deg_{x_i}(v)$ there exists $j$ with $\deg_{x_j}(u)<\deg_{x_j}(v)$ and $x_j(u/x_i)\in\mathcal{G}(I)$. Let $I_{\langle j\rangle}$ be the monomial ideal generated by the monomials of degree $j$ belonging to $I$. We say that $I$ is componentwise polymatroidal if $I_{\langle j\rangle}$ is polymatroidal for all $j$. Componentwise polymatroidal ideals are vertex splittable, see \cite[Theorem 3.1]{F-Polym} or \cite[Proposition 2]{CF3}.
\end{Examples}\bigskip

The following is an important property of vertex splittable ideals.
\begin{Proposition}\label{Prop:Saha}
	\textup{\cite[Theorem 4.15]{BMS24}} Let $I\subset S$ be an equigenerated vertex splittable monomial ideal. Then $\vstab(I)=1$ and $\v(I^k)=\alpha(I)k-1$ for all $k\ge1$.
\end{Proposition}\medskip

This result is no longer valid if $I$ is not equigenerated.
\begin{Example}
	\rm Let $I=(x_1^2,x_1x_2,x_1x_3^2,x_3^3)\subset K[x_1,x_2,x_3]$. Then $I$ is vertex splittable. Indeed, $I=x_1(x_1,x_2,x_3^2)+(x_3^3)$ is a vertex splitting, as $I_1=(x_1,x_2,x_3^2)$, $I_2=(x_3^3)$ are clearly vertex splittable and $I_2\subset I_1$. Notice that $I$ is not generated in a single degree. The \textit{Macaulay2} \cite{GDS} package \texttt{VNumber} \cite{FSPack} reveals that $\v(I^k)=2k$ for all $k\gg0$.
\end{Example}\bigskip

Combining Proposition \ref{Prop:Saha} with Theorem \ref{Thm:DisjointSupports} we obtain
\begin{Corollary}
	Let $I_1,\dots,I_t\subset S$ be equigenerated vertex splittable monomial ideals. Then, for all $k\ge1$,
	$$
	\v((\sum_{j=1}^tI_j)^k)\ =\ (\min_{1\le j\le t}\alpha(I_j))k+(\sum_{j=1}^t\alpha(I_j)-\min_{1\le j\le t}\alpha(I_j)-t).
	$$
\end{Corollary}

In view of Example \ref{ex:Pedro} and Corollary \ref{Cor:Pedro}(c) we expect:
\begin{Conjecture}
	Let $f_1,\dots,f_m$ be a homogeneous regular sequence on $S$, and let $f_i=g_{i,1}^{a_{i,1}}\cdots g_{i,b_i}^{a_{i,b_i}}$ be the decomposition of $f_i$ into irreducible polynomials, for $i\in\{1,\dots,m\}$. Set $I=(f_1,\dots,f_m)$. Then, for all $k\ge1$,
	$$
	\v(I^k)\ =\ \alpha(I)k+\sum_{i=1}^m(\deg f_i-\max_{1\le j\le b_i}\deg g_{i,j})-\alpha(I).
	$$
\end{Conjecture}

\textit{Acknowledgment.} We thank Kamalesh Saha for his careful reading of the manuscript, and his suggestions.


\begin{thebibliography}{99}
		
		\bibitem{ASS23} S.B. Ambhore, K. Saha, I. Sengupta, \textit{The $\v$-Number of Binomial Edge Ideals}, 2023, preprint \url{arXiv:2304.06416}
		
		\bibitem{BHR12} S. Bayati, J. Herzog, G. Rinaldo, \textit{On the stable set of associated prime ideals of a monomial ideal}, Arch. Math., 98(2012), 213-217.
		
		\bibitem{BHR11} S. Bayati, J. Herzog, G. Rinaldo, \textit{A routine to compute the stable set of associated prime ideals of a monomial ideal}, \url{http://ww2.unime.it/algebra/rinaldo/stableset/}, 2011.
		
		\bibitem{BHT15} S. Beyarslan, H.T. H\`a, T.N. Trung, \textit{Regularity of powers of forests and cycles}, J. Algebraic Combin. 42 (2015), no. 4, 1077–1095, DOI 10.1007/s10801-015-0617-y. MR3417259
		
		\bibitem{BM23} P. Biswas, M. Mandal, \textit{A study of $v$-number for some monomial ideals}, 2023, arXiv preprint \url{arXiv:2308.08604}.
		
		\bibitem{BMS24} P. Biswas, M. Mandal, K. Saha, \textit{Asymptotic behaviour and stability index of v-numbers of graded ideals}, 2024, preprint \url{arXiv:2402.16583}.
		
		\bibitem{B79} M. Brodmann, \textit{Asymptotic stability of $\textup{Ass}(M/I^nM)$}, Proc. Am. Math. Soc., 74(1979), 16--18
		
		\bibitem{Conca23} A. Conca. \textit{A note on the v-invariant}, Proceedings of the American Mathematical Society, 2024, 152(6), pp. 2349--2351
		
		\bibitem{CSTVV20} S.M. Cooper, A. Seceleanu, S.O. Toh\u{a}neanu, M. Vaz Pinto, R.H. Villarreal. \textit{Generalized minimum distance functions and algebraic invariants of Geramita ideals}. Advances in Applied Mathematics, 112:101940, 2020.
		
		\bibitem{CF} M. Crupi, A. Ficarra, \textit{Minimal resolutions of vector-spread Borel ideals}. Analele Stiintifice ale Universitatii Ovidius Constanta, Seria Matematicat, 2023, 31(2), 71--84.
		
		\bibitem{CF1} M. Crupi, A. Ficarra, \textit{Very well--covered graphs by Betti splittings}, J. Algebra {\bf 629}(2023) 76--108. https://doi.org/10.1016/j.jalgebra.2023.03.033.
		
		\bibitem{CF2} M. Crupi, A. Ficarra, \textit{Very well-covered graphs via the Rees algebra}, Mediterranean Journal of Mathematics 21.4 (2024): 135.
		
		\bibitem{CF3} M. Crupi, A. Ficarra, \textit{Cohen--Macaulayness of vertex splittable monomial ideals}, Mathematics 2024, 12(6), 912; https://doi.org/10.3390/math12060912
		
		\bibitem{DMNB23} A. De Stefani, J. Montaño, L. Núñez-Betancourt. \textit{Frobenius methods in combinatorics}. São Paulo Journal of Mathematical Sciences 17.1 (2023): 387-429.
		
		\bibitem{DJS24} D. Dey, A. V. Jayanthan, K. Saha, \textit{On the $\v$-number of binomial edge ideals of some classes of graphs}, 2024, preprint \url{https://arxiv.org/abs/2405.15354}
		
		\bibitem{Eis94} D. Eisenbud. \textit{Commutative Algebra with a View Toward Algebraic Geometry}. Graduate Texts in Mathematics 150, Springer-Verlag, 1994.  
		
		\bibitem{F-Vect} A. Ficarra, \textit{Vector-spread monomial ideals and Eliahou--Kervaire type resolutions}, J. Algebra {\bf615} (2023), 170--204.
		
		\bibitem{F-Polym} A. Ficarra, \textit{Shellability of componentwise discrete polymatroids}, 2023, preprint \url{https://arxiv.org/abs/2312.13006}.
		
		\bibitem{F2023} A. Ficarra, \textit{Simon's conjecture and the $\v$-number of monomial ideals}, Collectanea Mathematica (2024): 1-16. https://doi.org/10.1007/s13348-024-00441-z
		
		\bibitem{FS2} A. Ficarra, E. Sgroi, \textit{Asymptotic behaviour of the $\v$-number of homogeneous ideals}, 2023, preprint \url{https://arxiv.org/abs/2306.14243}
		
		\bibitem{FSPack} A. Ficarra, E. Sgroi, \texttt{VNumber}, \textit{Macaulay2 Package} available at \url{https://github.com/EmanueleSgroi/VNumber}, 2024.
		
		\bibitem{FSPackA} A. Ficarra, E. Sgroi, \textit{Asymptotic behaviour of integer programming and the $\v$-function of a graded filtration}, 2024, preprint \url{https://arxiv.org/abs/2403.08435}
		
		\bibitem{Fior24} L. Fiorindo, \textit{A theorem on the asymptotic behaviour of the generalised $\v$-number}, 2024, preprint \url{arxiv:2401.17815}.
		
		\bibitem{Ghosh24} L. Fiorindo, D. Ghosh, \textit{On the asymptotic behaviour of the Vasconcelos invariant for graded modules}, 2024, preprint \url{arxiv:2401.16358}.
		
		\bibitem{GDS} D.~R.~Grayson, M.~E.~Stillman. {\em Macaulay2, a software system for research in algebraic geometry}. Available at \url{http://www.math.uiuc.edu/Macaulay2}.
		
		\bibitem{HNTT20} H.T. H\`a, H.D. Nguyen, N.V. Trung, T.N. Trung, \textit{Symbolic powers of sums of ideals}. Math. Z. 294 (2020), 1499--1520. 1, 2, 4
		
		\bibitem{HTT16} H.T. H\`a, N. V. Trung, T. N. Trung, \textit{Depth and regularity of powers of sums of ideals}, Math. Z. (2016) 282:819--838. DOI 10.1007/s00209-015-1566-9. MR3473645
		
		\bibitem{HH2011} J.~Herzog, T.~Hibi. \emph{Monomial ideals}, Graduate texts in Mathematics {\bf 260}, Springer, 2011.
		
		\bibitem{HRR} J. Herzog, M. Rahimbeigi, T. Romer, \textit{Classes of cut ideals and their Betti numbers}. Sao Palo J. Math. Sci. {\bf 17}, 172--187 (2023) \url{https://doi.org/10.1007/s40863-022-00325-9}
		
		\bibitem{KS19} J. Kim, I. Swanson, \textit{Many associated primes of powers of prime ideals}, Journal of Pure and Applied Algebra 223 (2019), 4888–4900.
		
		\bibitem{KS23} N. Kotal, K. Saha, \textit{On the v-number of Gorenstein ideals and Frobenius powers}, 2023, preprint \url{https://arxiv.org/abs/2311.04136}.
		
		\bibitem{KNS24} M. Kumar, R. Nanduri, K. Saha, \textit{The slope of v-function and Waldschmidt Constant}, 2024, preprint \url{https://arxiv.org/abs/2404.00493}
		
		\bibitem{MKA16} S. Moradi, F. Khosh-Ahang. \textit{On vertex decomposable simplicial complexes and their Alexander duals}. Math. Scand., 118(1):43–56, 2016.
		
		\bibitem{M2023} S. Moradi, \textit{Normal Rees algebras arising from vertex decomposable simplicial complexes}, 2023, preprint \url{arxiv.org/abs/2311.15135}.
		
		\bibitem{NguyenT} H.D. Nguyen, Q.H. Tran, \textit{Powers of sums and their associated primes}, Pacific J. Math, Vol. 316, No. 1, 2022.
		
		\bibitem{R1976} L. J. Ratliff Jr., \textit{On prime divisors of $I^n$, $n$ large}, Mich. Math. J. 23 (1976), 337--352.
		
		\bibitem{S2023} K. Saha, \textit{The $\v$-number and Castelnuovo-Mumford regularity of cover ideals of graphs}, International Mathematics Research Notices, Volume 2024, Issue 11, June 2024, Pages 9010–9019, https://doi.org/10.1093/imrn/rnad277
		
		\bibitem{SS20} K. Saha, I. Sengupta, \textit{The $\v$-number of monomial ideals}. J Algebr Comb 56, 903--927 (2022) https://doi.org/10.1007/s10801-022-01137-y
		
		\bibitem{VS24} A. Vanmathi, P. Sarkar, \textit{$v$-numbers of symbolic power filtrations}, 2024, preprint \url{https://arxiv.org/abs/2403.09175}.
\end{thebibliography}
\end{document}